 \documentclass[12pt]{amsart}
\RequirePackage[colorlinks,citecolor=blue,urlcolor=blue,linkcolor=blue]{hyperref}
\usepackage[top=1 in,bottom=1 in, left=1 in, right= 1 in, includehead,includefoot]{geometry}

 \usepackage{tikz}
 \usetikzlibrary{patterns}
 \usepackage[numbers,sort&compress]{natbib}
\usepackage{amsmath,amsthm,amsfonts,amssymb}
\usepackage{soul}
 \usepackage{url,enumerate,soul}

\def\fddto{\xrightarrow{\textit{f.d.d.}}}
\newcommand{\ind}{{\bf 1}}

\def\inddd#1{{\ind}_{\left\{#1\right\}}}

\newcommand{\proba}{\mathbb P}
\newcommand{\esp}{{\mathbb E}}

\newcommand{\inv}{^{-1}}
\newcommand{\cov}{{\rm{Cov}}}

\newcommand{\eqnh}{\begin{eqnarray*}}
\newcommand{\eqne}{\end{eqnarray*}}
\newcommand{\eqnhn}{\begin{eqnarray}}
\newcommand{\eqnen}{\end{eqnarray}}
\newcommand{\equh}{\begin{equation}}
\newcommand{\eque}{\end{equation}}

\def\summ#1#2#3{\sum_{#1 = #2}^{#3}}
\def\prodd#1#2#3{\prod_{#1 = #2}^{#3}}

\def\topp#1{^{(#1)}}

\def\nn#1{{\left\|#1\right\|}}

\def\abs#1{\left|#1\right|}

\def\ccbb#1{\left\{#1\right\}}
\def\sccbb#1{\{#1\}}
\def\pp#1{\left(#1\right)}
\def\spp#1{(#1)}

\def\mmid{\;\middle\vert\;}

\def\floor#1{\left\lfloor #1 \right\rfloor}

\def\vv#1{{\boldsymbol #1}}

\def\qmand{\quad\mbox{ and }\quad}

\def\qmwith{\quad\mbox{ with }\quad}
\def\mfa{\mbox{ for all }}

\def\wt#1{\widetilde{#1}}
\def\wb#1{\overline{#1}}
\def\what#1{\widehat{#1}}

\def\limn{\lim_{n\to\infty}}

\def\limsupn{\limsup_{n\to\infty}}

\def\weakto{\Rightarrow}
\newcommand\aswto{\stackrel{a.s.w.}\to}

\def\Z{{\mathbb Z}}

\def\R{{\mathbb R}}

\def\N{{\mathbb N}}

\def\BB{{\mathbb B}}

\renewcommand{\d}{{\rm d}}

\newtheorem{theorem}{Theorem}[section]

\newtheorem{lemma}[theorem]{Lemma}
\newtheorem{proposition}[theorem]{Proposition}

\theoremstyle{definition}

\theoremstyle{remark}
\newtheorem{remark}[theorem]{Remark}
\theoremstyle{remark}

\newcommand{\calB}{{\mathcal B}}

\newcommand{\calG}{{\mathcal G}}

\newcommand{\calM}{\mathcal{M}}

\def\<{\langle}
\def\>{\rangle}

\newcommand{\rmTL}{{\rm TL}}
\newcommand{\rmRW}{{\rm RW}}
\newcommand{\rmBerW}{{\rm BerW}}

\numberwithin{equation}{section}

 \usepackage{framed,pgf}
\newcounter{oldeq}
\newcounter{usesofarxiv}
 \newcommand{\arxiv}[1]{
\setcounter{oldeq}{\value{equation}}
 \addtocounter{usesofarxiv}{1}
 \setcounter{equation}{0}
\def\theoldeq{\theequation}
\def\theequation{x-\arabic{usesofarxiv}.\arabic{equation}}
\def\theequation{\arabic{section}.\arabic{usesofarxiv}.\arabic{equation}}
\def\theequation{\thesection.\arabic{usesofarxiv}.\arabic{equation}}
  \colorlet{shadecolor}{gray!10}
{\footnotesize
\begin{shaded}#1
\end{shaded}
   \setcounter{equation}{\value{oldeq}}
\numberwithin{equation}{section}
}}
\renewcommand{\arxiv}[1]{}

\author{W\l odzimierz Bryc}
\address
{
W\l odzimierz Bryc\\
Department of Mathematical Sciences\\
University of Cincinnati\\
2815 Commons Way\\
Cincinnati, OH, 45221-0025, USA.
}
\email{wlodzimierz.bryc@uc.edu}

\author{Joseph Najnudel}
\address
{
Joseph Najnudel\\
School of Mathematics\\
University of Bristol\\
Bristol, United Kingdom.
}
\email{joseph.najnudel@bristol.ac.uk}

\author{Yizao Wang}
\address
{
Yizao Wang\\
Department of Mathematical Sciences\\
University of Cincinnati\\
2815 Commons Way\\
Cincinnati, OH, 45221-0025, USA.
}
\email{yizao.wang@uc.edu}

\begin{document}\sloppy
\title[Limit fluctuations of open TASEP on the coexistence line]{Limit fluctuations of stationary measure of totally asymmetric simple exclusion process with open boundaries on the coexistence line}

\begin{abstract}%
We describe limit fluctuations of the height function for the open TASEP on the coexistence line under the stationary measure. It is known that the height function satisfies a law of large numbers as the number of sites $n$ goes to infinity which at the coexistence line is exotic in the sense that the first-order limit is random. Here, we study  the functional central limit theorem:  we show that with a random centering and normalized by $\sqrt n$, the second-order limit of the height functions is a (random) mixture of two independent Brownian motions.
\end{abstract}

\date{\today \jobname.tex}

\maketitle

\arxiv{This is an expanded version of the manuscript.}
\section{Introduction} %
The totally asymmetric simple exclusion process (TASEP) with open boundaries, commonly referred to as open TASEP, is a fundamental stochastic model in non-equilibrium statistical physics. Its behavior, %
particularly the steady-state density profile,
is strongly shaped by the boundary conditions, making it a rich and compelling system for analysis. The model is defined as a continuous-time %
 irreducible Markov process on the finite state space $\{0,1\}^n$, where each configuration encodes the occupancy of $n$ lattice sites, with $1$ denoting an occupied site and $0$ an empty one.

Informally, open TASEP describes particles moving from left to right along a one-dimensional lattice of sites $\{1, \dots, n\}$, with open boundaries that permit %
particles to enter and exit.
Each particle attempts to jump to its immediate right neighbor at rate 1, provided that the site lies within $\{1, \dots, n\}$ and is unoccupied. Additionally, particles may enter the system at the leftmost site (site 1) at rate $\alpha > 0$ if it is empty, and may exit from the rightmost site (site $n$) at rate $\beta > 0$ if it is occupied. See Figure~\ref{fig:openASEP} for an illustration.

\begin{figure}[ht]
\centering
\begin{tikzpicture}[scale=0.7]
\draw[thick] (0.9, 0) -- (10.1, 0);
\foreach \x in {1, ..., 10} {
	\draw[gray] (\x, 0.085) -- (\x, -0.085);
}
\draw (10,-.1) node[below]{$n$};
\draw (1,-.1)  node[below]{$1$};
\fill[thick] (3, 0) circle(0.2);
\fill[thick] (6, 0) circle(0.2);
\fill[thick] (7, 0) circle(0.2);
\fill[thick] (10, 0) circle(0.2);
\draw[thick, ->] (3, 0.3)  to[bend left] node[midway, above]{$1$} (4, 0.3);
\draw[thick, ->] (7, 0.3) to[bend left] node[midway, above]{$1$} (8, 0.3);
\draw[thick, ->] (10, 0.3) to[bend left] node[midway, above]{$\beta$} (11, 0.3);
\draw[thick, ->] (-0.08, 0.5) to[bend left] node[midway, above]{$\alpha$} (0.9, 0.4);

\end{tikzpicture}
\caption{Jump rates in the open TASEP with 4 sites occupied.
}
\label{fig:openASEP}
\end{figure}
The unique stationary measure of the TASEP,
often referred to as the {\em steady state} in the physics literature,
has been extensively studied. It is well known that the large-scale behavior of this measure depends on the boundary parameters $\alpha$ and $\beta$, as captured by the celebrated phase diagram first described in %
\cite{derrida92exact,derrida93exact}.
In particular, along the so-called {\em coexistence line}, where $\alpha = \beta < 1/2$, the low entry and exit rates balance to produce a high-density traffic jam that forms at a random location. In the macroscopic limit, this jam location is uniformly distributed across the system; see \cite[page 293]{schutz93phase} for physics background and
 \cite{wang24askey} for mathematical statement  \eqref{eq:LLN}.

Our goal is to understand the fluctuations around this random limiting behavior. More precisely, let $(\tau_{n,1},\dots,\tau_{n,n})$ denote a configuration drawn from the stationary measure of the open TASEP with $n$ sites. We study the scaling limit of the associated height function defined by the %
partial-sum
 process
\[
h_n(t):=\summ k1{\floor {nt}}\tau_{n,k}, n\in\N, \; t\in[0,1],
\]
as $n \to \infty$.
Here,  anticipating a limit process with continuous Brownian  fluctuations, we use $t\in [0,1]$ as
a spatial variable indexing sites; some authors prefer to use $x$ instead to emphasize this interpretation.

A convenient reparameterization %
commonly used in the literature
is to restrict attention to $\alpha, \beta < 1$ and define \equh\label{eq:defining ABCD} a = \frac{1 - \alpha}{\alpha} > 0, \quad \text{and} \quad b = \frac{1 - \beta}{\beta} > 0. \eque %
A phase transition in terms of the limit fluctuations of $h_n(t)$ appropriately normalized has been known. Namely, the entire range of parameters is divided into:
\begin{enumerate}[(a)]
\item maximal current phase: $a,b<1$, %
\item high density phase: %
$b>a, b>1$,
\item low density phase: %
$a>b, a>1$.
\end{enumerate}
The limit fluctuations for all the three phases and also the boundary of maximal current phase
($a=1,b\le 1$ and $a\le 1,b=1$)
 have been characterized in \citep{derrida04asymmetric,bryc19limit,wang24askey}.
 In fact, all these results have been established for the general model of asymmetric simple exclusion process with open boundaries, commonly referred to as open ASEP, which contains open TASEP as a special case. For overview of ASEP, see for example~\citep{derrida07nonequilibrium,corwin22some,blythe07nonequilibrium}. In this paper we focus on open TASEP.

 A typical limit theorem takes the form as follows
\[
\ccbb{\frac1{\sqrt n}(h_n(t) - \floor{nt}\mu ) }_{t\in[0,1]}\fddto \ccbb{Z_t}_{t\in[0,1]},
\]
 where both the centering term $\mu$ and the limiting process $Z$ depend on the position of  parameters $a,b$  within the phase diagram.
The centering $\mu$ is known to be $1/2$ in the maximal current phase and its boundary, %
$b/(1+b)$
 in the high density phase, and
 $1/(1+a)$ in the low density phase.
 See Figure \ref{fig:phase} for an illustration of the phase diagram and the corresponding limit fluctuations $Z$.
  \begin{remark}Our notations are different from those in \citep{bryc19limit,wang24askey}, where $(A,C) = (b,a)$ was used. The choice of $A,C$ therein was more convenient when working with Askey--Wilson orthogonal polynomials associated to open ASEP \citep{bryc17asymmetric}. Our approach does not rely on this connection. The choice $(a,b)$ is more natural in general in the sense that the alphabetic order corresponds to the left-to-right order ($a$ and $b$ describe the behaviors of the left and right boundaries, respectively). \end{remark}

\begin{figure}[ht!]
\begin{tikzpicture}[scale=0.96]
 \draw[->] (5,5) to (5,10.2);
 \draw[->] (5.,5) to (11.2,5);
   \draw[-, dashed] (5,8) to (8,8);
   \draw[-, dashed] (8,8) to (8,5);
   \draw[-] (8,8) to (10,10);
   \node [left] at (5,8) {\scriptsize$1$};
   \node[below] at (8,5) {\scriptsize $1$};
     \node [below] at (11,5) %
     {$b$};
   \node [left] at (5,10) %
   {$a$};
    \node [above] at (6.5,9) {LD};
    \node [below] at (10,6.5) {HD};
     \node [below] at (6.5,6.5) {MC};
     \node at (10,10) {\small coexistence line };
 \draw[->] (4.5,8) to [out=-45,in=-135] (6,8);
  \node [above] at (3.8,7.9) {\small $\frac{1}{2\sqrt{2}}(\mathbb B+\widehat{\mathbb B}^{me})$};
  \node[left] at (8.2,4.5) {\small $\frac{1}{2\sqrt{2}}(\mathbb B+{\mathbb B}^{me})$};
 \draw[->] (8,4.5) to [out =45,in= -45] (8,6) ;
  \draw[-] (8,4.9) to (8,5.1);
   \draw[-] (4.9,8) to (5.1,8);
 \node [below] at (5,5) {\scriptsize$(0,0)$};
     \node [above] at (4.1,8.8){ \small %
     $\frac{\sqrt a}{1+a}\mathbb B$
     };
     \draw[->] (4.5,9) to [out=-45,in=-135] (6,9);
     \draw[->] (10,4.5) to [out =45,in= -45] (10,6) ;%
     \node [below] at (10,4.6) { \small %
     $\frac{\sqrt b}{1+b}\mathbb B$};
     \draw[->] (4.5,6.1) to [out=-45,in=-135] (6,6.18);
  \node [above] at (3.8,6) {\small $\frac{1}{2\sqrt{2}}(\mathbb B+ {\mathbb B}^{ex})$};
  \draw[->] (10,8.5) to [out=-90,in=-35] (8,8);
  \node at (10,8.7) {\small $\frac{1}{2}\mathbb B$};
\end{tikzpicture}
    \caption{
Phase diagram for the
limit
fluctuations of the hight function of open TASEP under stationary measure. 
LD, HD, and MC respectively stand for the low density, high density and maximal current phases.
Processes $\mathbb{B}$, ${\mathbb B}^{ex}$, ${\mathbb B}^{me}$ and $\widehat{\mathbb B}^{me}$ respectively stand for the Brownian motion, excursion, meander, and reversed meander; see, for example, \cite{bryc19limit} for their definitions. Processes %
in the sums are assumed to be independent. }
    \label{fig:phase}
\end{figure}
The only unsolved limit theorem for the fluctuations is on the {\em coexistence line}, which is the boundary between high and low density phases, corresponding to $a = b>1$. The situation here is delicate. In \citep{wang24askey}, it was shown that %
on the coexistence line, the first-order limit theorem has a random limit. Namely, it was shown there the convergence of finite-dimensional distributions of
\equh\label{eq:LLN}
\ccbb{\frac1{n}h_n(t) }_{t\in[0,1]}%
\fddto
 \ccbb{(t\wedge U)\frac1{1+a} + (t-U)_+\frac a{1+a}}_{t\in[0,1]}
\eque
as $n\to\infty$ where here and throughout we let $U$ denote a uniform random variable over $(0,1)$.
This result was established in \citep{wang24askey} for more general asymmetric simple exclusion process (ASEP),   where particles can also move to the left with rate $q<1$,  can leave the leftmost site 1 at  rate $\gamma\geq 0$ and can arrive at the rightmost site $n$ at rate $\delta\geq 0$.  %

Our contribution is to elaborate the convergence in \eqref{eq:LLN} by proving the second-order fluctuation limit, {\em in the special case of open TASEP}. Since the first-order limit is random, the centering for the second-order limit theorem has to be random, %
too. For this purpose, we write $s_{n,j} :=\tau_{n,1}+\cdots+\tau_{n,j}$ (with $s_{n,0}:=0$)  and introduce
\[
\tau^*_n:=\min\ccbb{j=0,\dots,n: s_{n,j} - \frac j2 = \min_{k=0,\dots,n} \pp{s_{n,k}-\frac k2}  }.
\]
Our main result is the following. %
Throughout, we let `$\weakto$' denote convergence in distribution.
\begin{theorem}\label{thm:1}
When $a=b>1$, we have
\[
\frac{\tau_n^*}n \weakto U,
\]
and
\begin{multline*}
\ccbb{\frac1{\sqrt n}\pp{h_n(t) - (\floor{nt}\wedge \tau^*_n)\frac1{1+a}  - (\floor{nt}-\tau^*_n)_+\frac a{1+a}}}_{t\in[0,1]}\\
\weakto \ccbb{\sigma_a\pp{\BB_{t\wedge U}+\BB'_{(t-U)_+}}}_{t\in[0,1]}
\end{multline*}
in $D[0,1]$ as $n\to\infty$, where $\BB$ and $\BB'$ are two standard Brownian motions independent from $U$, and
\[
\sigma_a = \frac{\sqrt {a}}{1+a}.
\]
\end{theorem}
Our proof is based on a representation of the stationary measure of the open TASEP in terms of two interacting random walks, coupled explicitly through a Radon–Nikodym derivative with respect to the product measure. This representation, introduced in \citep{bryc24two},  was inspired by the recent work of \citet{barraquand24stationary}, who developed a two-line representation for a different model in integrable probability. It
differs from previous random walk-based methods for analyzing the stationary measure of the open TASEP 
  \cite{barraquand23stationaryJPA,derrida04asymmetric,himwich24stationary}. Notably, it remains valid along the coexistence line, which constitutes the central focus of this study.

The key idea of our approach is to analyze the difference between the two random walks described in \citep{bryc24two}, conditioning on the location of the minimum of this difference process. To the right of the minimum, the process behaves like a random walk with positive drift conditioned to stay above its minimum. To the left, we obtain the time-reversal of a similar conditioned process. The trajectories of these two walks are indicted with arrows in Figure~\ref{FigH}.

A classical result of %
\citet{iglehart74functional} shows that conditioning a random walk with positive drift to stay positive does not alter its fluctuation behavior (see Lemma~\ref{lem:aux}); hence, both segments converge to independent Brownian motions in the scaling limit. Notably, our proof does not rely on \eqref{eq:LLN}; rather, it yields \eqref{eq:LLN} as a consequence, and in fact establishes it in the stronger mode of convergence in the Skorokhod space $D[0,1]$.

 This approach is distinctly more probabilistic in nature than recent analytical methods used to study fluctuation limits for open ASEP (e.g., \citep{bryc19limit,corwin24stationary,wang24askey}), which rely heavily on the matrix product ansatz introduced by \citet{derrida93exact} and its deep connection to the Askey–Wilson process \citep{bryc17asymmetric,uchiyama04asymmetric}.

At the same time, it is worth pointing out that while the two-line representation we work with here is strictly limited to the open TASEP at this moment, we expect the result to hold for general open ASEP.
Despite recent advances \cite{bryc25stationary}, it is still unclear how we should approach this question.
While the aforementioned approach via Askey--Wilson representation can be applied to establish various limit theorems on fluctuations of stationary measure of general open ASEP \citep{bryc19limit,bryc23asymmetric,corwin24stationary,wang25asymmetric}, on the coexistence line the approach is limited to a first-order limit theorem as in \eqref{eq:LLN}.
The limitation of the Askey--Wilson representation is that the representation does not provide an accessible description of  the law of $\tau_n^*$.

{\em The paper is organized as follows.} %
Section \ref{sec:TL} introduces the two-line representation of the stationary measure for the open TASEP, as recently developed in \citep{bryc24two}. We also present a slightly modified version in Theorem \ref{thm:BZ'}, which is more convenient for our later analysis.
In Section \ref{sec:Denisov}, we offer an alternative representation of the stationary measure for the open TASEP, stated in Theorem \ref{thm:TL denisov}. This formulation builds on the two-line representation and incorporates Denisov’s decomposition of a random walk.
Section \ref{sec:aux} presents an auxiliary result: a functional central limit theorem for a biased random walk.
Finally, in Section \ref{sec:proof}, we restate Theorem \ref{thm:1} as Theorem \ref{thm:1'} using the representation from Theorem \ref{thm:TL denisov}, and provide its proof.

\section{Two-line representation}\label{sec:TL}
Recall  that  in \eqref{eq:defining ABCD} we reparametrized the two boundary parameters  {$\alpha,\beta<1$} of open TASEP  via
\[%
a = \frac{1-\alpha}\alpha>0, \qmand b = \frac{1-\beta}\beta>0.
\]%
In this section, we fix $a,b>0$,  $n\in\N$, and %
we use %
boldface %
$\vv v = (v_1,\dots,v_n)$
to denote %
an $n$-dimensional vector in general. Consider
\[
\Omega_{\rmTL,n} = \{0,1\}^{n}\times \{0,1\}^n,
\]
 and represent each element in $\Omega_{\rmTL,n}$ as $(\vv\omega\topp1,\vv\omega\topp2)$. Set $s_0:=0$ and
 \[
 s_j\topp i \equiv s_j\topp i(\vv\omega\topp i):= \omega_1\topp i+\cdots+\omega_j\topp i,\quad j=1,\dots,n, i=1,2.
 \]
 Let $P_{\rmBerW,n}$ %
 denote
 the law of $n$-step vector of a Bernoulli random walk with parameter $1/2$; that is, the uniform probability measure on $\{0,1\}^n$.
Let $P_{\rmTL,n,a,b}$ be the probability measure on $\Omega_{\rmTL,n}$ determined by
\[%
\frac{\d P_{\rmTL,n,a,b}}{\d P^{\otimes2}_{{\rm BerW},n}}\pp{\vv\omega\topp1,\vv\omega\topp2} \propto \frac{b^{s_n\topp1-s_n\topp 2}}{(ab)^{\min_{j=0,\dots,n}(s_j\topp 1-s_j\topp2 )}}, \quad \pp{\vv\omega\topp1,\vv\omega\topp2}\in\Omega_{\rmTL,n}.
\]%
We are in particular interested in the marginal law of $P_{\rmTL,n,a,b}$ on the first argument. That is, with
\[
\pi(\vv\omega\topp1,\vv\omega\topp2) := \vv\omega\topp1,
\] the law $P_{\rmTL,n,a,b}\circ\pi\inv$. The following result was established in \citep{bryc24two}.
\begin{theorem}\label{thm:BZ}
The law of $P_{\rmTL,n,a,b}\circ\pi\inv$  is the law of $(\tau_{n,1},\dots,\tau_{n,n})$, that is, the stationary measure of the open TASEP of size $n$ with parameters $\alpha = 1/(1+a),\beta = 1/(1+b)$.
\end{theorem}
For our purpose, we shall also work with the induced measure
\[
 \wt P_{\rmTL,n,a,b}:=P_{\rmTL,n,a,b}\circ\Phi\inv
\]
by the function $\Phi:\Omega_{\rmTL,n}\to\wt\Omega_{\rmTL,n}:=\{0,1\}^n\times \{-1,0,1\}^n$ given by
\[
 \Phi\pp{\vv\omega\topp1,\vv\omega\topp2} = \pp{\vv\omega\topp1,\vv\omega\topp1-\vv\omega\topp2}.
\]
 (Note that $\Phi$ is one-to-one, but not onto; in particular $\wt P_{\rmTL,n,a,b}$ does not have full support on $\wt\Omega_{\rmTL,n}$.)
 We shall also work with random walks with steps taking values from $\{-1,0,1\}$ and law of steps as $\nu_1 = (1/4)\delta_{-1}+(1/4)\delta_1+(1/2)\delta_0$ (steps $\pm1$ with probability $1/4$ each and $0$ with probability $1/2$).

We shall work with the following variation of Theorem \ref{thm:BZ}.
For
$\vv\omega = (\omega_1,\dots,\omega_n) \in \{-1,0,1\}^n$, write $s_j\equiv s_j(\vv\omega) :=\omega_1+\cdots+\omega_j, j=1,\dots,n$ and $s_0 = 0$.
\begin{theorem}\label{thm:BZ'}
For all $a,b>0$,
\equh\label{eq:Bayesian}
\wt P_{\rmTL,n,a,b}\pp{\vv\omega\topp1,\vv\omega} = P_{\rmRW,n,a,b}(\vv\omega)\prodd j1n q\pp{\omega_j\topp1\mmid \omega_j},
\eque
with $P_{\rmRW,n,a,b}$ determined by
\equh\label{eq:Denisov0}
\frac{\d P_{\rmRW,n,a,b}}{\d \nu_1^{\otimes n}}(\vv\omega) =\frac1{C_{n,a,b}} \frac{b^{s_n}}{(ab)^{\min_{j=0,\dots,n}s_j}}, \quad \vv\omega\in\{-1,0,1\}^n,
\eque
for some normalizing constant $C_{n,a,b}$ (see \eqref{eq:C_n,a,b} below)
and
\[
q(\omega'\mid \omega) := \begin{cases}
1, & \mbox{ if }(\omega',\omega) = (1,1)  \mbox{ or } (0,-1),\\
1/2, & \mbox{ if } (\omega',\omega) = (1,0) \mbox{ or } (0,0),\\
0, & \mbox{ otherwise.}
\end{cases}
\]
\end{theorem}
\begin{proof}
The result follows from checking the Bayesian formula for conditional probability carefully.
\end{proof}
\section{An extension of Denisov's decomposition}\label{sec:Denisov}
We establish a representation  of the height function of open TASEP in terms of random variables  (instead of a formula as a function on the sample space) that is convenient for our analysis later.

 The representation we introduce is notationally lengthy, so we begin with an overview before presenting the precise definitions below. Our construction is inspired by Denisov's decomposition of random walks \citep{denisov83random}, which we refer to as {\em Denisov's representation} for open TASEP. Specifically, we construct a sample $(\vv S'_n,\vv S_n)$ and ultimately demonstrate that the law of their increment processes coincides exactly with
  $\wt P_{\rmTL,n,a,b}$.

In this section, we explicitly indicate the dimensions of possibly random vectors, which are denoted by boldface letters.
When the vector represents a path of length $\ell$, it is denoted by $\vv s_\ell = (0,s_{\ell,1},\dots,s_{\ell,\ell})$ with the initial value 0 included (so the total length of $\vv s_\ell$ is $\ell+1$). When the vector represents the increments of the path, it is denoted by $\Delta\vv s = (s_{\ell,1},s_{\ell,2}-s_{\ell,1},\dots,s_{\ell,\ell}-s_{\ell,\ell-1})$.

In words, to sample a path of the bivariate random walk $(\vv S'_n,\vv S_n)$ we proceed as follows.
\begin{enumerate}[(1)]
\item We start by sampling the first moment when the process $\vv S_n$ reaches the minimum of the first $n$ steps, denoted by $T_n$. Then unless $T_n = 0$, we have that the $T_n$-th step of $\vv S_n$ is necessarily $-1$. (We do not discuss case $T_n =0$ in this overview here, which only requires a slight modification of the description below.)
\item We then sample two conditionally independent random vectors (with length $T_n-1$ and $n-T_n$ respectively) denoted by
$\vv L_{T_n-1}$ and $\vv R_{n-T_n}$,
respectively.
\item %
We then set the path of $\vv S_n$ before the $(T_n-1)$-th step as the path starting from zero with increments corresponding to those of $\vv L_{T_n-1}$ {\em in the time-reversal order}, and set the path of $\vv S_n$ after the $T_n$-th step as the path represented by $\vv R_{n-T_n}$.

\color{black}
\item
 A desired sample path of $\vv S_n'$ is obtained similarly by first sampling ${\vv L}_{T_n-1}'$ and ${\vv R}_{n-T_n}'$, which are two random vectors coupled with $\vv L_{T_n-1}$ and $\vv R_{n-T_n}$ respectively, and then concatenating the two to the $T_n$-th step (which is 0 in place of $-1$ this time).
\end{enumerate}

In words, the two processes $\vv L_{T_n-1}$ and $\vv R_{n-T_n}$ represent the left and right processes, respectively, of $\vv S_n$ starting from the minimum location $T_n$, and the two are conditionally independent given $T_n$. Similarly, $\vv L_{T_n-1}'$ and $\vv R_{n-T_n}'$ are the left and right processes of $\vv S_n'$ respectively starting {\em also} from $T_n$.

Now we provide a complete description. For this purpose we need notations for random walks with step law
parameterized by $a>0$ as
\[
\nu_a(\omega) := \begin{cases}
\displaystyle\frac a{a+a\inv+2}, & \mbox{ if } \omega = 1,\\\\
\displaystyle\frac 2{a+a\inv+2}, & \mbox{ if } \omega = 0,\\\\
\displaystyle\frac {a\inv}{a+a\inv+2}, & \mbox{ if } \omega = -1,
\end{cases}
\quad \omega \in\{-1,0,1\}.
\]
The mean of each step is $\int \omega \nu_a(\d\omega) = (a-a\inv)/(a+a\inv+2) = (a-1)/(a+1)$.
Then, the $\ell$-step vector of such a random walk has law $\nu_a^{\otimes \ell}$ on $\{-1,0,1\}^\ell$.

We first give the formula for the law of $T_n$. We let $p\topp a_\ell$ denote the probability that a random walk with step law $\nu_a$ does not reach $-1$ during the first $\ell$ steps. Then, we consider the law
\equh\label{eq:T_n}
\proba(T_n = m) = \begin{cases}
\displaystyle\frac a{4C_{n,a,b}}w_a^{m-1}w_b^{n-m}p\topp a_{m-1}p\topp b_{n-m}, & \mbox{ if }
m=1,\dots,n,\\\\
\displaystyle\frac1{C_{n,a,b}}w_b^np\topp b_n, &\mbox{ if } m = 0,
\end{cases}
\eque
with
\[
w_a := \frac a4+\frac 1{4a}+\frac12,
\]
and
\equh\label{eq:C_n,a,b}
C_{n,a,b} = \sum_{m=1}^{n}\frac a4 w_a^{m-1}p_{m-1}\topp a w_b^{n-m}p_{n-m}\topp b + w_b^np\topp b_n.
\eque

We next introduce a family of processes
$\vv L_{m}$ and $\vv R_{m}$ indexed by $m\in\N_0$.
For each $m\in\N_0$, we let
\begin{align*}
\vv L_m & :=\ccbb{L_{m,j}}_{j=0,\dots,m},\\
 \vv R_m&:=\ccbb{R_{m,j}}_{j=0,\dots,m},
\end{align*}
be stochastic processes satisfying the following.
\begin{enumerate}[(i)]
\item The process $\vv L_m$ has the law of the $m$-step path of a random walk starting from 0 with step law $\nu_a$, conditioned to remain non-negative up to step $m$.
\item The process $\vv R_m$ has the law of the $m$-step path of a random walk starting from 0 with step law $\nu_b$, conditioned to remain non-negative up to step $m$.
\end{enumerate}
We assume further that
$\{T_n\}_{n\in\N}$,
$\{\vv L_m\}_{m\in\N_0}$ and $\{\vv R_m\}_{m\in\N_0}$ are independent.

Given an $n$-step path $\vv s = (0,s_1,\dots,s_n)$ (starting from 0), set the $n$-step vector as
\[
\Delta \vv s := (s_1,s_2-s_1,\dots,s_n-s_{n-1}).
\]
We next introduce the processes $\vv L_m', \vv R_m'$, coupled with $\vv L_m$ and $\vv R_m$ respectively as follows
\[
\Delta L_{m,j}':=\begin{cases}
0, & \mbox{ if } \Delta L_{m,j} = 1,\\
-1, & \mbox{ if } \Delta L_{m,j} = -1,\\
-\xi_{m,j}^\leftarrow, & \mbox{ if } \Delta L_{m,j} = 0,
\end{cases}
\qmand
\Delta R_{m,j}':=\begin{cases}
1, & \mbox{ if } \Delta R_{m,j} = 1,\\
0, & \mbox{ if } \Delta R_{m,j} = -1,\\
\xi_{m,j}^\rightarrow, & \mbox{ if } \Delta R_{m,j} = 0,
\end{cases}
\]
where $\{\xi_{m,j}^\leftarrow\}_{m\in\N,j=1,\dots,m}, \{\xi_{m,j}^\rightarrow\}_{m\in\N, j=1,\dots,m}$ are i.i.d.~Bernoulli random variables with parameter $1/2$, independent from
$\{T_n\}_{n\in\N}$,  $\{\vv L_m\}_{m\in\N_0}$, and $\{\vv R_m\}_{m\in\N_0}$.
Note that   this construction consists of three independent sequences
$\{T_n\}_{n\in\N}$, $\{(\vv L_m,\vv L_{m}')\}_{m\in \N_0}$, and $\{(\vv R_m,\vv R_{m}')\}_{m\in\N_0}$, and they are all needed in the proof of Theorem \ref{thm:1}. For Theorem \ref{thm:TL denisov} in this section we shall need, for each $n\in\N$ fixed, only the construction of $T_n$, $\{(\vv L_m,\vv L_{m}')\}_{m=0,\dots,n-1}$, and $\{(\vv R_m,\vv R_{m}')\}_{m=0,\dots,n}$.

We then introduce the concatenated process
 \[
{\vv S}_n:=\begin{cases}\vv L_{T_n-1}\odot \vv R_{n-T_n}, & \mbox{ if } T_n = 1,\dots,n,\\
\vv R_n, & \mbox{ if } T_n = 0,
\end{cases}\quad n\in\N,
\] with
\begin{multline*}
\vv L_{m-1}\odot \vv R_{n-m}:=\Big(0, \underbrace{L_{m-1,m-2}-L_{m-1,m-1}, \dots,L_{m-1,1}-L_{m-1,m-1},-L_{m-1,m-1}}_{m-1~{\rm terms}},\\
-L_{m-1,m-1}-1,
\underbrace{-L_{m-1,m-1}-1+R_{n-m,1},\dots,-L_{m-1,m-1}-1+R_{n-m,n-m}}_{n-m~\rm terms}\Big), m=1,\dots,n.%
\end{multline*}
(In particular, the position of the process at $m$-th step is $-L_{m-1,m-1}-1$.)
Similarly, we set
\[
{\vv S}'_n:=\begin{cases}{\vv L}_{T_n-1}'\odot' {\vv R}_{n-T_n}', & \mbox{ if } T_n = 1,\dots,n,\\
\vv R'_n, & \mbox{ if } T_n = 0,
\end{cases}
\]
where
\begin{multline}
{\vv L}_{m-1}'\odot' {\vv R}_{n-m}'
:=\Big(0, \underbrace{L'_{m-1,m-2}-L'_{m-1,m-1}, \dots,L'_{m-1,1}-L'_{m-1,m-1},-L'_{m-1,m-1}}_{m-1~{\rm terms}},\\
-L'_{m-1,m-1},
\underbrace{-L'_{m-1,m-1}+R'_{n-m,1},\dots,
-L'_{m-1,m-1}
+R'_{n-m,n-m}}_{n-m~\rm terms}\Big), m=1,\dots,n.\label{eq:wt S^1}
\end{multline}
(Note that the $m$-th step of the path $\vv L_{m-1}\odot \vv R_{n-m}$ is $-1$, while the $m$-th step of the path $\vv L'_{m-1}\odot'\vv R'_{n-m}$ is $0$.)

\begin{figure}[hbt]
    \begin{tikzpicture}[scale=.65]
 \draw[-] (.9,0) to (15.1,0);

   \draw[-] (1,-.1) to (1,0.08);
  \draw[-] (2,-.1) to (2,0.08);
  \draw[-] (3,-.1) to (3,0.08);
  \draw[-] (4,-.1) to (4,0.08);
  \draw[-] (5,-.1) to (5,0.08);
  \draw[-] (6,-.1) to (6,0.08);
  \draw[-] (7,-.1) to (7,0.08);
  \draw[-] (8,-.1) to (8,0.08);
  \draw[-] (9,-.1) to (9,0.08);
  \draw[-] (10,-.1) to (10,0.08);
  \draw[-] (11,-.1) to (11,0.08);
  \draw[-] (12,-.1) to (12,0.08);
    \draw[-] (13,-.1) to (13,0.08);
      \draw[-] (14,-.1) to (14,0.08);
      \draw[-] (15,-.1) to (15,0.08);

\draw [fill] (1,0) circle [radius=0.08];
\draw[-,thick,dashed] (1,0) to (2,1);
\draw [fill] (2,1) circle [radius=0.08];
\draw[-,thick,dashed] (2,1) to (3,2);
\draw [fill] (3,2) circle [radius=0.08];
\draw[-,thick,dashed] (3,2) to (4,2);
\draw [fill] (4,2) circle [radius=0.08];
 \draw[-,thick,dashed] (4,2) to (5,2);
\draw [fill] (5,2) circle [radius=0.08];
 \draw[-,thick,dashed] (5,2) to (6,2);
\draw [fill] (6,2) circle [radius=0.08];
 \draw[-,thick,dashed] (6,2) to (7,3);
\draw [fill] (7,3) circle [radius=0.08];
 \draw[-,thick,dashed] (7,3) to (8,3);
\draw [fill] (8,3) circle [radius=0.08];
 \draw[-,thick,dashed] (8,3) to (9,3);
\draw [fill] (9,3) circle [radius=0.08];
 \draw[-,thick,dashed] (9,3) to (10,4);
\draw [fill] (10,4) circle [radius=0.08];
 \draw[-,thick,dashed] (10,4) to (11,5);
\draw [fill] (11,5) circle [radius=0.08];
 \draw[-,thick,dashed] (11,5) to (12,5);
\draw [fill] (12,5) circle [radius=0.08];
 \draw[-,thick,dashed] (11,5) to (12,5);
  \draw[-,thick,dashed] (12,5) to (13,6);
\draw [fill] (13,6) circle [radius=0.08];
  \draw[-,thick,dashed] (13,6) to (14,7);
\draw [fill] (14,7) circle [radius=0.08];
  \draw[-,thick,dashed] (14,7) to (15,8);
\draw [fill] (15,8) circle [radius=0.08];

\draw [fill] (9,-1) circle [radius=0.08];
\draw[-,dotted,thick] (9,0) to (9,-2);
\node[below] at (9,-2) { \footnotesize $T_n'$};
\draw [fill] (9,0) circle [radius=0.08];

\draw[-,thick,dashed] (1,0) to (2,.5);
\draw [fill] (2,.5) circle [radius=0.08];
\draw[-,thick,dashed] (2,.5) to (3,1);
\draw [fill] (3,1) circle [radius=0.08];
\draw[-,thick,dashed] (3,1) to (4,.5);
\draw [fill] (4,.5) circle [radius=0.08];
 \draw[-,thick,dashed] (4,.5) to (5,0);
\draw [fill] (5,0) circle [radius=0.08];
 \draw[-,thick,dashed] (5,0) to (6,-.5);
\draw [fill] (6,-.5) circle [radius=0.08];
 \draw[-,thick,dashed] (6,-.5) to (7,0);
\draw [fill] (7,0) circle [radius=0.08];
 \draw[-,thick,dashed] (7,0) to (8,-.5);
\draw [fill] (8,-.5) circle [radius=0.08];
 \draw[-,thick,dashed] (8,-.5) to (9,-1);
\draw [fill] (9,-1) circle [radius=0.08];
\draw (9,-1) circle [radius = 0.2];
 \draw[-,thick,dashed] (9,-1) to (10,-.5);
\draw [fill] (10,-.5) circle [radius=0.08];
 \draw[-,thick,dashed] (10,-.5) to (11,0);
\draw [fill] (11,0) circle [radius=0.08];
 \draw[-,thick,dashed] (11,0) to (12,-.5);
\draw [fill] (12,-.5) circle [radius=0.08];
  \draw[-,thick,dashed] (12,-.5) to (13,0);
\draw [fill] (13,0) circle [radius=0.08];
  \draw[-,thick,dashed] (13,0) to (14,.5);
\draw [fill] (14,.5) circle [radius=0.08];
  \draw[-,thick,dashed] (14,.5) to (15,1);
\draw [fill] (15,8) circle [radius=0.08];

   \node[below] at (1,0) {\footnotesize  $0$};
      \node[below] at (2,0) { \footnotesize $ 1$};
       \node[below] at (3,0) { \footnotesize $2$};

             \node[below] at (15,0) { \footnotesize $n$};

\node[right] at  (15,8) { \tiny $\vv S_n' = \vv L_{T_n-1}'\odot'\vv R'_{n-T_n}$};
\draw[<-,blue,thick] (1,0) to (2,0);
\draw[<-,blue,thick] (2,0) to (3,1);
\draw[<-,blue,thick] (3,1) to (4,0);
\draw[<-,blue,thick] (4,0) to (5,-1);
\draw[-,thick] (5,-1) to (6,-2);
\draw[-,dotted,thick] (6,0) to (6,-2);
\node[below] at (6,-2) { \footnotesize $T_n=m$};
\draw [fill] (6,0) circle [radius=0.08];

\draw[->,blue,thick] (6,-2) to (7,-1);
\draw[->,blue,thick] (7,-1) to (8,-1);
\draw[->,blue,thick] (8,-1) to (9,-2);
\draw[->,blue,thick] (9,-2) to (10,-1);
\draw[->,blue,thick] (10,-1) to (11,-1);
\draw[->,blue,thick] (11,-1) to (12,-1);
\draw[->,blue,thick] (12,-1) to (13,0);
\draw[->,blue,thick] (13,0) to (14,1);
\draw[->,blue,thick] (14,1) to (15,2);
\draw [fill,blue] (5,-1) circle [radius=0.08];
\draw [fill] (6,-2) circle [radius=0.08];
\draw (6,-2) circle [radius=0.2];

\node[right] at  (15,2.3) { \tiny $\vv S_n = \vv L_{m-T_n}\odot \vv R_{n-T_n}$};
\node[right] at  (15,1) { \tiny $\what{\vv S}_n'$};

\draw[dotted,thick] (15,2) to (17,2);
\draw[dotted,thick] (6,-2) to (17,-2);
\draw[->,thick] (16,-2) to (16,2);
\node[right] at (16,-0.5) {\tiny $R_{n-T_n,n-T_n}$};
\draw[dotted,thick] (0,-1) to (5,-1);
\draw[dotted,thick] (0,0) to (1,0);
\draw[->,thick] (0.5, -1) to (0.5, 0);
\node[left] at (0.5,-0.5) {\tiny $L_{T_n-1,T_n-1}$};
\draw[dotted,thick] (4,-2) to (6,-2);
\draw[->,thick] (4.5, -2) to (4.5, -1);
\node[left] at (4.5,-1.5) {\tiny $1$};
\draw[dotted,thick] (15,0) to (17,0);
\draw[->,thick] (16.5,0) to (16.5,2);
\node[right] at (16.5,1) {$\substack{S_{n,n} = R_{n-T_n,n-T_n}\\\quad\quad - L_{n,T_n-1}-1}$};

\end{tikzpicture}
  \caption{An illustration of the construction.
   With $n=14, T_n = m = 5$,
the left random walk is  $\vv L_{n,T_n-1} = (0,1,2,1,1)$ and
the right random walk is $\vv R_{n,n-T_n} = (0,1,1,0,1,1,1,2,3,4)$. The paths $\vv S_n$, $\vv S_n'$, and $\what{\vv S}_n'$ are plotted.  In this example, $T_n' = 8$.
  }\label{FigH}
\end{figure}

We have the following results.

\begin{theorem}\label{thm:TL denisov}
 The law $\wt P_{\rmTL,n,a,b}$ on $\{0,1\}^n\times\{-1,0,1\}^n$
is the law of the  increment processes $(\Delta{\vv S}_n',\Delta{\vv S}_n)$.  In particular, $\Delta{\vv S}'_n$ has the law
 of  stationary measure of open TASEP with parameters $\alpha = 1/(1+a),\beta = 1/(1+b)$.
\end{theorem}

We shall also consider $\what{\vv S}_n' = (0,\what S_{n,1}',\dots,\what S_{n,n}')$ with $\what S_{n,j}' = S_{n,j}' - j/2, j=0,\dots,n$, and
\equh\label{eq:T_n'}
T_n':=\min\ccbb{m=0,\dots,n:\what S_{n,m}' = \min_{j=0,\dots,n} \what S'_{n,j}}.
\eque
Now, the random vector $(\tau_{n,1},\dots,\tau_{n,n},\tau_n^*)$ has the same law as $(\Delta S_{n,1}',\dots,\Delta S_{n,n}',T_n')$.
Figure \ref{FigH} provides an illustration.

To prove Theorem \ref{thm:TL denisov}, it suffices to establish the following. %
\begin{proposition}\label{prop:denisov}
The law of $P_{\rmRW,n,a,b}$ on $\{-1,0,1\}^n$ as determined by \eqref{eq:Denisov0} is the law of the process $(\Delta S_1,\dots,\Delta  S_n)$.
\end{proposition}
Indeed, once Proposition \ref{prop:denisov} is established, one readily checks that the coupling between ${\vv S}'_n$ and ${\vv S_n}$ corresponds exactly to the formula \eqref{eq:Bayesian} in Theorem \ref{thm:BZ'}, which completes the proof of Theorem~\ref{thm:TL denisov}.

The rest of this section is devoted to the proof of Proposition \ref{prop:denisov}.
\begin{proof}
Throughout we fix $n\in\N$. Let
\[
\calM_\ell:=\ccbb{\vv\omega\in\{-1,0,1\}^\ell:\min_{j=0,\dots,\ell}s_j(\vv\omega) = 0}
\] denote the space of all paths of length $\ell$ that remain non-negative.

For each path from $\{-1,0,1\}^n$, we decompose it at its minimum position as follows. Set
\equh\label{eq:t_n}
t_n\equiv t_n(\vv\omega) :=
\min\ccbb{k=0,\dots,n:s_{k}(\vv\omega) = \min_{j=0,\dots,n}s_j(\vv\omega)}, \vv\omega\in\{-1,0,1\}^n.
\eque
So, the minimum of the path is first achieved at $t_n$. Moreover, when $t_n>0$ it necessarily follows that $\omega_{t_n}=-1$. We set accordingly two paths
\begin{align*}
\vv\omega^\leftarrow\equiv\vv\omega^\leftarrow(\vv\omega)&:=\pp{-\omega_{t_n-1},\dots,-\omega_1}\in\calM_{t_n-1}, \quad t_n\ge 2,\\
\vv\omega^\rightarrow\equiv\vv\omega^\rightarrow(\vv\omega)&:=\pp{\omega_{t_n+1},\dots,\omega_n}\in\calM_{n-t_n}, \quad t_n\le n-1.
\end{align*}
Note that $\omega_{t_n}$ is excluded from both sequences. Introduce accordingly
\begin{align*}
s^\leftarrow_{t_n-1} & := \omega^\leftarrow_1+\cdots+\omega^\leftarrow_{t_n-1},\\
s^\rightarrow_{n-t_n} & := \omega^\rightarrow_1+\cdots+\omega^\rightarrow_{n-t_n}.
\end{align*}
Set $s^\leftarrow_0 = s^\rightarrow_0 = 0$.
Note that when $t_n = 0$, we shall only need $s^\rightarrow_n$ below.
Then,
starting from \eqref{eq:Denisov0}, we have, when $t_n \ne 0$,
\begin{align*}
P_{\rmRW,n,a,b}(\vv\omega) &=\frac1{C_{n,a,b}}\nu_1^{\otimes n}(\vv\omega)\frac{b^{s_n}}{(ab)^{\min_{j=0,\dots,n}s_j}} =  \frac1{C_{n,a,b}}\nu_1^{\otimes n}(\vv\omega) \frac{b^{s_{n-t_n}^\rightarrow -
s_{t_n-1}^\leftarrow-1}}{(ab)^{-s_{t_n-1}^\leftarrow-1}}\\
&=  \frac1{C_{n,a,b}}\nu_1^{\otimes n}(\vv\omega) a^{s_{t_n-1}^\leftarrow+1}b^{s_{n-t_n}^\rightarrow} \\
& = \frac1{C_{n,a,b}}
\frac{a}4 \pp{\nu_1^{\otimes(t_n-1)}(\vv\omega^\leftarrow)a^{s_{t_n-1}^\leftarrow}}\pp{\nu_1^{\otimes(n-t_n)}(\vv\omega^\rightarrow)b^{s_{n-t_n}^\leftarrow}}.
\end{align*}
Note that we used the fact that $\omega_{t_n}=-1$  and $\nu_1(-1) = 1/4$
  in the last step.
Similarly, when $t_n = 0$ we have
\[
P_{\rmRW,n,a,b}(\vv\omega) = \frac1{C_{n,a,b}}\nu_1^{\otimes n}(\vv\omega)b^{s_n^\rightarrow}.
\]
Noticing that $a^\omega \nu_1(\omega)=w_a\nu_a(\omega)$, we have
\begin{align*}
\nu_1^{\otimes m}(\vv\omega_m)a^{\omega_{m,1}+\cdots+\omega_{m,m}} &= \pp{\frac a4}^{\summ j1m \inddd{\omega_{m,j} = 1}}\pp{\frac1{4a}}^{\summ j1m \inddd{\omega_{m,j} = -1}}\pp{\frac12}^{\summ j1m \inddd{\omega_{m,j} = 0}} \\
& = w_a^m\nu_a^{\otimes m}(\vv\omega_m) \quad\vv\omega_m\in\calM_m.
\end{align*}

Therefore, we have arrived at
\equh\label{eq:PRW}
P_{\rmRW,n,a,b}(\vv\omega) = \begin{cases}
\displaystyle\frac a{4C_{n,a,b}}w_a^{t_n-1}\nu_a^{\otimes(t_n-1)}(\vv\omega^\leftarrow) w_b^{n-t_n}\nu_b^{\otimes(n-t_n)}(\vv\omega^\rightarrow), & \mbox{ if } t_n>0,\\\\
\displaystyle\frac1{C_{n,a,b}}w_b^n\nu_b^{\otimes n}(\vv\omega),& \mbox{ if } t_n = 0.
\end{cases}
\eque

It remains to recognize that the above is the law of $(\Delta S_1,\dots,\Delta S_n)$ with $\vv S_n$ constructed at the beginning of this section. To see this, since $\sum_{\vv\omega_m\in\calM_m}\nu_a^{\otimes m}(\vv\omega_m) = p_m\topp a$ is the probability that a random walk with step law $\nu_a$ does not reach to $-1$ during the first $m$ steps, we first recognize
\[
\frac1{p_m\topp a}\nu^{\otimes m}_a(\cdot) \mbox{ restricted to $\calM_m$}
\]  as the conditional law of $m$-step vector of a random walk with step law $\nu_a$ given that it does not reach $-1$ during the first $m$ steps. Now, we write the expression in \eqref{eq:PRW} when $t_n>0$ as
\begin{multline*}
\frac a{4C_{n,a,b}}w_a^{t_n-1}\nu_a^{\otimes(t_n-1)}(\vv\omega^\leftarrow) w_b^{n-t_n}\nu_b^{\otimes(n-t_n)}(\vv\omega^\rightarrow) \\
=
\frac a{4C_{n,a,b}}w_a^{t_n-1}w_b^{n-t_n}p_{t_n-1}\topp a p_{n-t_n}\topp b\frac{\nu_a^{\otimes(t_n-1)}(\vv\omega^\leftarrow)}{p_{t_n-1}\topp a}
\frac{\nu_b^{\otimes(n-t_n)}(\vv\omega^\rightarrow)}{p_{n-t_n}\topp b}.
\end{multline*}
On the right-hand side above, we recognize the product of the laws of %
increments of
$\vv L_{t_n-1}$ and $\vv R_{n-t_n}$, and the multiplicative weight $(a/4)w_a^{t_n-1}w_b^{n-t_n}p_{t_n-1}\topp ap_{n-t_n}\topp b$ matches exactly the weight in the formula \eqref{eq:T_n} of $T_n$. The case $t_n = 0$ can be interpreted in a similar way. This completes the proof.
\end{proof}
\section{An auxiliary result for biased random walks}\label{sec:aux}
In  Denisov's representation introduced in Section \ref{sec:Denisov}, the left and right processes are biased random walks conditioned to not to reach $-1$ for a finite number steps. As an auxiliary result we investigate two functional central limit theorems here.
\begin{lemma}\label{lem:aux}
Let $\{X_n\}_{n\in\N}$ be i.i.d.~random variables with law $\nu_a$ with $a>1$. Set
\[
X_n':=\begin{cases}
1, & \mbox{ if } X_n = 1,\\
0, & \mbox{ if } X_n = -1,\\
\xi_n, & \mbox{ if } X_n = 0,
\end{cases}
n\in\N,
\]
where $\{\xi_n\}_{n\in\N}$ are i.i.d.~Bernoulli random variables with parameter $1/2$ independent from $\{X_n\}_{n\in\N}$. Set $W_n:=X_1+\cdots+X_n, W_n':=X_1'+\cdots+X_n', n\in\N$ and $W_0 = W_0':=0$.
\begin{enumerate}[(i)]
\item Then,
\equh\label{eq:WW'}
\pp{\ccbb{\frac1{\sqrt n}\pp{W_{\floor{nt}}- \floor{nt}\frac{a-1}{a+1}}}_{t\in[0,1]},
\ccbb{\frac1{\sqrt n}\pp{W_{\floor{nt}}'-\floor{nt}\frac a{a+1}}}_{t\in[0,1]}}
\weakto\vec\BB_a
\eque
in $D[0,1]^2$  as $n\to\infty$,
where $\vec\BB_a = \ccbb{\BB_a(t),\BB'_a(t)}_{t\in[0,1]}$ is %
a bivariate Brownian motion that satisfies
\[
\cov\pp{\vec\BB_a(1),\vec\BB_a(1) } %
= \esp \pp{\vec \BB_a(1)\vec\BB_a(1)^T}
 = \frac1{(a+1)^2}\pp{\begin{array}{cc}
2a & a\\
a & a
\end{array}}.
\]
\item Assume in addition $a>1$.
Consider the stopping time $\tau_{-1}:=\min\{\ell\in\N: W_\ell = -1\}$. If for each $n$ the left-hand side above is with respect to the conditional law given the event $\{\tau_{-1}>n\}$, then the same convergence \eqref{eq:WW'} holds.
\item Moreover, the tightness in both cases above is established with respect to the uniform topology (more precisely, for the second case see \eqref{eq:W_n tight4} below).
\end{enumerate}
\end{lemma}
\begin{proof}
One readily checks that
\[
\esp X_1 = \frac{a-1}{a+1}, \quad\esp X_1'= \frac a{a+1},
\]
and, with $\vec X = (X_1,X_1')$,
\[ \cov
\pp{\vec X,\vec X} = \frac1{(a+1)^2}\pp{\begin{array}{cc}
2a & a\\
a & a
\end{array}}.
\]
Part (i) follows from standard Donsker's theorem (for random walk in $\Z^2$).

We next prove part (ii).
For each $n\in\N, t\in[0,1]$, let $(W_n(t),W'_n(t))$ denote the random vector indexed by $t$ on the left-hand side of \eqref{eq:WW'}. Write
\equh\label{eq:vec W}
\vec {\vv W}_n = (\vv W_n,\vv W'_n) = \pp{\ccbb{W_n(t)}_{t\in[0,1]},\ccbb{W'_n(t)}_{t\in[0,1]}} \in D[0,1]^2.
\eque
Let $f:D[0,1]^2\to \R$ be a continuous and bounded function. The goal is to show
\equh\label{eq:goal}
\limn \esp \pp{f\pp{\vec{\vv W}_n}\mmid\tau_{-1}>n} = \esp  f\pp{\vec\BB_a}.
\eque
We follow the idea of \citet[Section 5]{iglehart74functional}. Recall that with $a>1$, we have $\proba(\tau_{-1} < \infty) = 1/a^2$. Write
\begin{multline}
\esp \pp{f\pp{\vec{\vv W}_n}\mmid\tau_{-1}>n}\\
 = \frac1{\proba(\tau_{-1}>n)}\pp{\esp f\pp{\vec{\vv W}_n} - \summ k1n \esp \pp{f\pp{\vec {\vv W}_n}\mmid \tau_{-1} = k}\proba(\tau_{-1} = k)}.\label{eq:0}
\end{multline}
Part (i) tells exactly that
\[
\limn \esp f\pp{\vec{\vv W}_n} = \esp f\pp{\vec\BB_a}.
\]
We also have
\equh\label{eq:D exercise}
\limn \esp f\pp{\vec{\vv W}_n\mmid \tau_{-1}=k} = \esp f\pp{\vec\BB_a}, \mfa k\in\N,
\eque
of which the proof is provided below.
 Then, it follows that
\equh
\lim_{K\to\infty}\limn\sum_{k=1}^K \esp \pp{f\pp{\vec {\vv W}_n}\mmid \tau_{-1} = k}\proba(\tau_{-1} = k) = \esp f\pp{\vec\BB_a}\proba(\tau_{-1}<\infty),\label{eq:1}
\eque
and
\equh
\lim_{K\to\infty}\limsupn\sum_{k=K+1}^n \esp \pp{f\pp{\vec {\vv W}_n}\mmid \tau_{-1} = k}\proba(\tau_{-1} = k) \le \lim_{K\to\infty} \nn f_\infty \proba(K<\tau_{-1}<\infty) = 0,  \label{eq:2}
\eque
where we used the continuity of measure in the last step.     Combining \eqref{eq:0}, \eqref{eq:1}, and \eqref{eq:2}, we have
\begin{align*}
\limn\esp \pp{f\pp{\vec{\vv W}_n}\mmid\tau_{-1}>n}
 &= \frac1{\proba(\tau_{-1}=\infty)}\pp{\esp f\pp{\vec\BB_a} - \esp f\pp{\vec\BB_a}\proba(\tau_{-1}<\infty)} \\
 &= \esp f\pp{\vec\BB_a},
\end{align*}
which completes the proof of \eqref{eq:goal}.

To complete the proof of part (ii) it remains to prove \eqref{eq:D exercise}. We first show the tightness. For a stochastic process $\eta = \{\eta_t\}_{t\in[0,1]}$, introduce its modulus of continuity
\[
\omega(\eta,\delta):= \sup_{s,t\in[0,1],|t-s|\le\delta}\abs{\eta_t-\eta_s}.
\]
We establish
\begin{align}
\lim_{\delta\downarrow0}\limsupn\proba\pp{\omega(\vv W_n,\delta)>\eta\mmid \tau_{-1} = k}&=0,\label{eq:W_n tight}\\
\lim_{\delta\downarrow0}\limsupn\proba\pp{\omega(\vv W'_n,\delta)>\eta\mid\tau_{-1}=k}&=0.\nonumber
\end{align}
We only show the first claim and the proof is the same for the second. Notice that we can write, for $k$ fixed and $n>k$,
\begin{align*}
\omega(\vv W_n,\delta) &\le {\sup_{s,t\in[0,k/n],|t-s|\le\delta}\abs{W_n(t)-W_n(s)} + \sup_{s,t\in[k/n,1],|t-s|\le\delta}\abs{W_n(t)-W_n(s)}}\\
& =: {\omega_{n,k}\topp 1(\vv W_n,\delta)+\omega_{n,k}\topp2(\vv W_n,\delta)}.
\end{align*}
It is clear that by construction
\[%
\omega_{n,k}\topp1\pp{\vv W_n,\delta}\le \frac{2k}{\sqrt n}\quad \mbox{ almost surely,}
\]%
and hence $\omega_{n,k}\topp1\pp{\vv W_n,\delta}\to 0$ almost surely. Next, notice that under the conditional law given $\tau_{-1} = k$, $W_n(k/n) = -1/\sqrt n$, and therefore
\equh\label{eq:W_n,k}
W_{n,k}(s):=W_n\pp{s+\frac kn}+\frac1{\sqrt n}, s\ge 0,
\eque is the normalized partial-sum process of a random walk with step law $\nu_a$ starting from 0 (independent from the value of $k\in\N$; we also extend the definition of $W_n(t)$ to $t>1$). Since, with $\vv W_{n,k} = \{W_{n,k}(s)\}_{s\in[0,1]}$,
\[
\omega_{n,k}\topp 2(\vv W_n,\delta)\le \omega(\vv W_{n,k},\delta),
\]
and it is well-known that (see \citep{billingsley99convergence} for details)
\[%
\lim_{\delta\downarrow0}\limsupn \proba\pp{\omega(\vv W_{n,k},\delta)>\eta\mmid\tau_{-1} =k} = 0,
\]%
 it follows that
\equh\label{eq:W_n tight2}
\lim_{\delta\downarrow 0}\limsup_{n\to\infty}\proba\pp{\omega_{n,k}\topp2(\vv W_n,\delta)>\eta\mmid \tau_{-1} = k} = 0.
\eque
We have proved the tightness condition \eqref{eq:W_n tight}.

We next show the marginal convergence.
The convergence of finite-dimensional distributions follows by essentially the same argument and is omitted. Fix $t\in(0,1]$ and $k\in\N$. We explain how to show the following: for all continuous and bounded functions $f:\R\to\R$,
\equh\label{eq:W_n CLT}
\limn\esp \pp{f\pp{W_n(t)}\mid \tau_{-1} = k} = \esp f\pp {\sigma_{a,t} Z} \qmwith \sigma_{a,t} = \frac{\sqrt {2at}}{1+a},
\eque
and $Z$ on the right-hand side is a standard normal random variable. Indeed, for $n$ large enough so that $\floor{nt}>k$, by \eqref{eq:W_n,k}  we have
\equh\label{eq:W_n tight3}
\abs{W_n(t)-W_{n,k}(t)}\le  \abs{W_{n,k}\pp{t-\frac kn}-W_{n,k}(t)}  + \frac1{\sqrt n}.
\eque
Under the conditional law given $\tau_{-1} = k$, the central limit theorem says $W_{n,k}(t)\weakto \sigma_{a,t} Z$. Then
\begin{align*}
\limsupn&~\proba\pp{\abs{W_n(t) - W_{n,k}(t)}>\eta\mmid\tau_{-1} = k}\\
& \le \limsupn\proba\pp{\abs{W_{n,k}\pp{t-\frac kn} - W_{n,k}(t)} >\eta\mmid\tau_{-1} = k}\\
& \le \limsupn \proba\pp{\omega\topp 2_{n,k}(\vv W_n,\delta)>\eta\mmid\tau_{-1}=k},
\end{align*}
for any $\delta>0$, and hence
thanks to  \eqref{eq:W_n tight2} we take $\delta\downarrow 0$ and arrive at
$W_n(t) - W_{n,k}(t)\to 0$ in probability as $n\to\infty$ (with respect to the conditional law given $\tau_{-1} = k$). We have proved \eqref{eq:W_n CLT} and hence part (ii).

For part (iii), the tightness for the process with respect to the original law (without any conditioning) is well-known. It suffices to show
\equh\label{eq:W_n tight4}
\lim_{\delta\downarrow 0}\limsupn\proba\pp{\omega(\vv W_n,\delta)>\eta\mmid\tau_{-1}>n} = 0.
\eque
Again following the idea behind \eqref{eq:0}, we have
\[
\lim_{\delta\downarrow0}\limsupn\proba\pp{\omega(\vv W_n,\delta)>\eta\mmid \tau_{-1}>n} \le \frac1{\proba(\tau_{-1} = \infty)}\lim_{\delta\downarrow0}\limsupn\proba(\omega(\vv W_n,\delta)>\eta) = 0,
\]
where in the last step we used the tightness of the process with respect to its original law.
\end{proof}

\section{Proof of the main result}\label{sec:proof}
Recall $T_n'$ given in \eqref{eq:T_n'}. Consider
\equh\label{eq:centering T_n'}
\wb S'_{n,j}:=\frac1{\sqrt n}\pp{S'_{n,j} - ((T'_n-1)\wedge j)\frac1{a+1} - (j-T'_n)_+\frac a{a+1}}, j=0,\dots,n.
\eque
Now, we restate Theorem \ref{thm:1} using the representation introduced in Theorem \ref{thm:TL denisov}.
\begin{theorem}\label{thm:1'}
With $a=b>1$, we have
\equh\label{eq:FCLT wb S}
\ccbb{\wb S'_{n,\floor{nt}}}_{t\in[0,1]}\weakto \ccbb{\sigma_a\BB_{t\wedge U} + \sigma_a\BB'_{(t-U)_+}}_{t\in[0,1]},
\eque
in $D[0,1]$ as $n\to\infty$, %
 where $U$ is a uniform random variable over $(0,1)$, $\{\BB_t\}_{t\in[0,1]}$ and $\{\BB_t'\}_{t\in[0,1]}$ are two independent Brownian motions also independent from $U$, and  \[\sigma_a :=\frac{\sqrt{a}}{a+1}.\]
\end{theorem}
Throughout, all convergence in $D[0,1]$ is with respect to the uniform topology.
We start with an overview of the proof.
\begin{enumerate}[(i)]
\item We first shall establish in Proposition \ref{prop:uniform} that ${T_n}/n\weakto U$ as $n\to\infty$.
\item Next, we show in Proposition \ref{prop:fluctuation T_n} that
\equh\label{eq:FCLT}
\pp{\pp{\wb {\vv L}_{T_n-1},\wb{\vv L}_{T_n-1}'},\pp{\wb {\vv R}_{n-T_n},\wb{\vv R}_{n-T_n}'},\frac{T_n}n}
\weakto\pp{\vec\BB_a\topp1,\vec\BB_a\topp2,U},
\eque
in $D[0,1]^2\times D[0,1]^2\times\R$, where $\wb {\vv L}_{T_n-1},\wb{\vv L}_{T_n-1}',\wb {\vv R}_{n-T_n},\wb{\vv R}_{n-T_n}'$ are certain normalized processes to be introduced below, and $\vec\BB_a\topp1,\vec\BB_a\topp2$ are i.i.d.~copies of the bivariate Brownian motion
\[
\vec\BB_a = \ccbb{\BB_a(t),\BB'_a(t)}_{t\in[0,1]}
\]
that satisfies
\[
\cov\pp{\vec\BB_a(1),\vec\BB_a(1)} %
=\esp\pp{\vec\BB_a(1)\vec\BB_a(1)^T}
= \frac1{(a+1)^2}\pp{\begin{array}{cc}
2a & a\\
a & a
\end{array}}.
\]
The processes $\vec\BB_a\topp1,\vec\BB_a\topp2$ are also independent from $U$.
\item The functional central limit theorem \eqref{eq:FCLT} immediately yields a functional central limit theorem for $S_{n,j}'$ normalized slightly differently from \eqref{eq:centering T_n'} (based on $T_n$ instead of $T_n'$).
From here, to prove the desired convergence, it suffices to show that $\{T_n-T_n'\}_{n\in\N}$ is tight, which is established in Lemma \ref{lemma:tightness}.
\end{enumerate}

We establish the three steps above in each of the following subsections respectively. Moreover, our approach is based on a coupled version of the construction introduced earlier, and in particular we will prove a stronger version of \eqref{eq:FCLT} and also a stronger version of \eqref{eq:FCLT wb S} (see Remark \ref{rem:asw version}) that elaborates the conditional independence structure in Proposition \ref{prop:fluctuation T_n}.
\subsection{Convergence of the minimum location}
We derive the limit law of $T_n$ in case $a=b>1$.

\begin{proposition}\label{prop:uniform}
Let $T_n$ %
have the law \eqref{eq:t_n} with $a=b>1$. Then,
\[
\frac{T_n}n\weakto U,
\]
as $n\to\infty$.
\end{proposition}\begin{proof}
For $x\in(0,1)$, we have
\equh\label{eq:T_n le}
\proba(T_n\le nx)  = \frac{aw_a^{n-1}}{4C_{n,a,a}}\sum_{m=1}^{\floor{nx}}p\topp a_{m-1}p\topp a_{n-m} + \frac1{C_{n,a,a}}w_a^np_n\topp a.
\eque
We know that
\[
\limn p\topp a_n = 1-\frac1{a^2},
\]
which is the probability that the biased random walk never visits $-1$. Notice that for any sequence $\{b_n\}_{n\in\N}$ such that $\limn b_n = b\in \R$, we have $n\inv\summ k1{\floor{nx}} b_kb_{n-k} = xb^2, x\in[0,1]$. It follows that the last term on the right-hand side of \eqref{eq:T_n le} tends to zero,
\[
\summ m1{\floor{nx}}p_{m-1}\topp ap_{n-m}\topp a\sim   nx \pp{1-\frac1{a^2}}^2,
\]
and hence (recalling the formula of $C_{n,a,a}$ in \eqref{eq:C_n,a,b})
\[
C_{n,a,a}\sim  \frac a4 n \cdot  w_a^{n-1}\pp{1-\frac1{a^2}}^2.
\]
We can now conclude that $\limn \proba(T_n\le nx) = x$ for all $x\in(0,1)$.
\end{proof}

\subsection{Convergence of the left and right processes}
We are interested in the following transformed processes (with different means):
\begin{align*}
\wb L_n(t) &:= \frac1{\sqrt n}\pp{L_{n,\floor{nt}} - \floor{nt}\frac{a-1}{a+1}},\\
\wb L_n'(t) &:= \frac1{\sqrt n}\pp{L_{n,\floor{nt}}' + \floor{nt}\frac1{a+1}},\\
\wb R_n(t) &:= \frac1{\sqrt n}\pp{R_{n,\floor{nt}} - \floor{nt}\frac{a-1}{a+1}},\\
\wb R_n'(t) &:= \frac1{\sqrt n}\pp{R'_{n,\floor{nt}} - \floor{nt}\frac a{a+1}}, \quad n\ge 1.
\end{align*}
When $n=0$ we set all the above processes to $0$ (a constant function indexed by $t\in[0,1]$). We write $\wb{\vv L}_n \equiv \{\wb L_n(t)\}_{t\in[0,1]}$ and similar notations for $\wb{\vv L}'_n, \wb{\vv R}_n, \wb{\vv R}'_n$, all viewed as random elements in $D[0,1]$.

We make use of the notion of almost surely weak convergence.
This will %
help us to exploit the conditional independence structure of the left and right processes given $T_n$.
We say a sequence of random elements $\{X_n\}_{n\in\N}$ taking values from a Polish space $(\mathcal X,d)$ converges almost surely weakly to another random element $X$ (also taking values from $(\mathcal X,d)$) with respect to $\calG$ if, for any countable convergence-determining class $\{f_i\}_{i\in\N}$ for convergence of probability measures on $(\mathcal X, \calB(\mathcal X))$ ($\calB(\mathcal X)$ is the Borel $\sigma$-algebra induced by the metric on $\mathcal X$),
\equh\label{eq:def aswto}
\limn \esp (f_i(X_n)\mid \calG) = \esp (f_i(X)\mid\calG) \mfa i\in\N,  \mbox{ almost surely.}
\eque
(The existence of countable convergence-determining class is known.)  In this case, we write
\[%
X_n \aswto X \mbox{ w.r.t.~$\calG$,}
\]%
as $n\to\infty$.
 Note that the above immediately implies that $X_n\weakto X$ is as $n\to\infty$. In addition, in practice, it suffices to establish \eqref{eq:def aswto} for any continuous and bounded function $f$. A nice reference on almost sure weak convergence that we found (on a slightly more general setup) is \citep[Section 4]{grubel16functional} and we refer to it for more details.

In order to make use of this notion of convergence,
by Skorokhod's representation theorem we assume in addition that in our construction of the processes in Section \ref{sec:Denisov}, %
the random sequence $\{T_n\}_{n\in\N}$ and $U$ are constructed on a common probability space such that
\[%
\limn \frac{T_n}n = U \mbox{ almost surely.}
\]%
In particular, now all processes are defined on a common probability space.
 Then throughout we set $\calG = \sigma(\{T_n\}_{n\in\N})$ and we shall need almost surely a weak convergence with respect to $\calG$ below.

\begin{proposition}\label{prop:fluctuation T_n}
Assume $a=b>1$. We have the following joint convergence:
\begin{equation}
  \label{conv0}
  \pp{\pp{\wb  {\vv L}_{T_n-1},\wb{\vv L}_{T_n-1}'},\pp{\wb {\vv R}_{n-T_n},\wb{\vv R}_{n-T_n}'}}
\aswto\pp{{\vec \BB}_a\topp1,\vec\BB_a\topp2} \mbox{ w.r.t.~$\calG$,}
\end{equation}
in $D[0,1]^2\times D[0,1]^2$ as $n\to\infty$, where $\vec\BB\topp1_a,\vec\BB\topp2_a$ are i.i.d.~copies of two-dimensional Brownian motion $\vec\BB_a$.
As a consequence, \eqref{eq:FCLT} holds.
\end{proposition}
\begin{proof}
  It follows from the construction that the sequences $\{T_n\}_{n\in\N}$, $\{\spp{\wb  {\vv L}_{n},\wb{\vv L}_{n}'}\}_{n\in\N_0}$ and $\sccbb{\spp{\wb  {\vv R}_{n},\wb{\vv R}_{n}'}}_{n\in\N_0}$ are independent.    For every $n=1,2,\dots$, processes
$\spp{\wb  {\vv L}_{n},\wb{\vv L}_{n}'}$ and $\spp{\wb  {\vv R}_{n},\wb{\vv R}_{n}'}$   have the same  law as the process  $\vec {\vv W}_n$
(see \eqref{eq:vec W})
 conditioned on $\{\tau_{-1}>n\}$  in the notation of Lemma \ref{lem:aux}.
Thus, by Lemma \ref{lem:aux},
$  \spp{\wb  {\vv L}_{n},\wb{\vv L}_{n}'}\weakto \vec\BB_a $ and $\spp{\wb  {\vv R}_{n},\wb{\vv R}_{n}'}\weakto \vec\BB_a$
in $D[0,1]$ as $n\to\infty$.

Independence of the three sequences implies that for each $n\in\N$,
 processes $\spp{\wb  {\vv L}_{T_n-1},\wb{\vv L}_{T_n-1}'}$ and $\spp{\wb {\vv R}_{n-T_n},\wb{\vv R}_{n-T_n}'}$ are conditionally independent  with respect to~$\calG$.
Then, by coupling $T_n(\omega)/n\to U(\omega)\in(0,1)$, so $T_n(\omega)\to \infty$ and $n-T_n(\omega)\to\infty$ almost surely.  This implies joint convergence \eqref{conv0} to the
 pair of i.i.d.~copies of the two-dimensional Brownian motion $\vec\BB_a$.
\end{proof}

\subsection{Proof of Theorem \ref{thm:1'}}

Following Proposition \ref{prop:fluctuation T_n}, we can now establish a functional central limit theorem for $\vv S_n'$.
Recall that we have seen in \eqref{eq:wt S^1} that
\[
{S}'_{n,j} = \pp{L'_{T_n-1,T_n-1-j}-L'_{T_n-1,T_n-1}}\inddd{j<T_n} + \pp{-L'_{T_n-1,T_n-1}+R'_{n-T_n,j-T_n}}\inddd{j\ge T_n}.
\]
We first consider the following normalization, where the centering depends on $T_n$ (which is slightly different from $\wb S_{n,j}$ of our interest):
\begin{align*}
\wt S_{n,j}'&:=\frac1{\sqrt n}\pp{S_{n,j}' - ((T_n-1)\wedge j)\frac1{a+1} - (j-T_n)_+\frac a{a+1}}%
\\
& = \sqrt{\frac{T_n-1}n}\pp{\wb L'_{T_n-1,T_n-1-j} - \wb L'_{T_n-1,T_n-1}}\inddd{j<T_n} \nonumber\\
& \quad + \pp{-\sqrt{\frac{T_n-1}n}\wb L'_{T_n-1,T_n-1} + \sqrt{\frac{n-T_n}n}\wb R'_{n-T_n,j-T_n}}\inddd{j\ge T_n}.\nonumber
\end{align*}
\begin{lemma}With notations above, we have
\equh\label{eq:FCLT wt S}
\ccbb{\wt S'_{n,\floor{nt}}}_{t\in[0,1]}\aswto \ccbb{\sigma_a\BB_{t\wedge U}+\sigma_a\BB_{(t-U)_+}}_{t\in[0,1]} \mbox{ w.r.t.~$\calG$}
\eque
in $D[0,1]$ as $n\to\infty$, where $\BB$ and $\BB'$ are two independent Brownian motions also independent from $U$.
\end{lemma}
\begin{proof}
We have
\equh
\wt S_{n,\floor{nt}}'
 = \sqrt{\frac{T_n-1}n}\wb L'_{T_n-1,T_n-1-\floor{nt}}\inddd{\floor{nt}<T_n} - \wb L'_{T_n-1,T_n-1} +  \sqrt{\frac{n-T_n}n}\wb R'_{n-T_n,\floor{nt}-T_n}\inddd{\floor{nt}\ge T_n}.\label{eq:3 terms}
\eque
Write
\begin{align*}
T_n-1-\floor{nt} &= (T_n-1)\pp{1-\frac{\floor{nt}}{T_n-1}},\\
\floor{nt}-T_n & = (n-T_n)\frac{\floor{nt}/n-T_n/n}{1-T_n/n}.\nonumber
\end{align*}
Formally  \eqref{eq:3 terms} and the fact that $T_n/n\to U$ almost surely suggest that
\equh\label{eq:aswto}
\wt S_{n,\floor{nt}}'
 \aswto\sqrt{U} \sigma_a\BB_{1-t/U}\inddd{t<U} -\sigma_a\BB_1 + \sqrt{1-U}\sigma_a\BB'_{(t-U)/(1-U)}\inddd{t>U}
 \eque
 with respect to $\calG$ as $n\to\infty$.
 Note that the right-hand side above equals in distribution to $\sigma_a\BB_{t\wedge U} + \sigma_a\BB'_{(t-U)_+}$ as desired.

We first prove \eqref{eq:aswto} for $t\in[0,1]$ fixed.
 For the sake of simplicity we illustrate how to prove \eqref{eq:aswto} by proving the convergence of the first term of $\wt S'_{n,\floor{nt}}$ in \eqref{eq:3 terms} only: the third term can be treated similarly, and the convergence of the second term follows from the central limit theorem. Namely, we show, for every continuous and bounded function $f:\R\to\R$,
\begin{align}
\esp& \pp{f\pp{\sqrt{\frac{T_n-1}n}\wb L'_{T_n-1,T_n-1-\floor{nt}}\inddd{\floor{nt}< T_n}}\mmid\calG} \nonumber\\
&=\esp \pp{f\pp{\sqrt{\frac{T_n-1}n}\wb L'_{T_n-1, (T_n-1)\pp{1-\frac{\floor{nt}}{T_n-1}}}\inddd{\floor{nt}\le T_n-1}}\mmid\calG} \nonumber\\
& \to \esp \pp{f\pp{\sqrt U \sigma_a\BB_{1-t/U}\inddd{t<U}}\mmid \calG}.\label{eq:?}
\end{align}

We justify \eqref{eq:?}. Throughout we fix $f$.
Introduce
\[
F_{n,m}(y):=\esp f\pp{\sqrt{\frac mn}\wb L'_{m,\floor{my}}\inddd{y\in[0,1]}}.
\]
Then, for any sequences $\{m_n\}_{n\in\N}, \{y_n\}_{n\in\N}$ such that $m_n/n\to u\in(0,1), y_n\le 1, y_n\to y\in[0,1]$ as $n\to\infty$, we have
\equh\label{eq:F}
\limn F_{n,m_n}(y_n) = F_u(y):=\esp f\pp{\sqrt u \sigma_a\BB_y\inddd{y\in[0,1]}}.
\eque
Indeed, when $y_n \equiv y$, the indicator function can be removed without changing the value of the expression, and the convergence is an immediate consequence of the weak convergence of $\wb L_{m,\floor{my}}\weakto \sigma_a\BB_y$. When $y_n\to y$ with nonconstant $\{y_n\}_{n\in\N}$, we first discuss $y\in(0,1]$. In this case, again the indicator function can be removed without changing the value of the expression for large $n$. We also need the tightness of convergence in $D[0,1]$ with respect to the uniform topology, more precisely the following criteria on the tightness (which is a re-write of \eqref{eq:W_n tight4} proved in Lemma \ref{lem:aux})
\equh\label{eq:L' tight}
\lim_{\delta\downarrow0}\limsup_{m\to\infty} \proba\pp{\sup_{s,t\in[0,1],|s-t|\le \delta}\abs{\wb L'_{m,\floor{ms}}-\wb L'_{m,\floor{mt}}}>\eta} = 0.
\eque
Then, one can prove \eqref{eq:F} with $y_n\to y\in(0,1]$ by the same argument after \eqref{eq:W_n tight3}, and we omit the details.
For $y_n\to y = 0$, it is possible that $y_n<0$, and in this case $F_{n,m_n}(y_n) = 0 = F_u(0)$; therefore it suffices to show that restricted to $y_n\ge 0, y_n\to y = 0$ the convergence \eqref{eq:F} holds, and this case follows again from the same argument above for $y_n\to y\in(0,1]$.

Then, with
\[
m_n(\omega):=T_n(\omega)-1, \quad y_n(\omega):=1-\frac{\floor{nt}}{T_n(\omega)-1}, \qmand y(\omega):=1-\frac t{U(\omega)},
\]
by \citep[Theorem 8.5]{kallenberg21foundations}, the desired convergence \eqref{eq:?} is the same as
\equh\label{eq:f}
F_{n,m_n(\omega)}(y_n(\omega))\to
\esp f\pp{\sqrt {U(\omega)}\sigma_a\BB_{1-t/U(\omega)}\inddd{t<U(\omega)}} \mbox{ almost every $\omega$}
\eque
as $n\to\infty$, which is a special case of \eqref{eq:F}.  We have thus proved \eqref{eq:f} and hence \eqref{eq:?}.

The argument above can be generalized to show that \eqref{eq:aswto} holds as the almost surely weak convergence of the finite-dimensional distributions of $\{\wt S'_{n,\floor{nt}}\}_{t\in[0,1]}$ to the corresponding limit for almost every $\omega\in\Omega$. Moreover, it is clear that we also have the tightness with respect to the uniform topology for almost every $\omega\in\Omega$ (essentially this is because $\{\wb{\vv L}'_m\}_{m\in\N}$ is tight in the sense of \eqref{eq:L' tight}, so is $\{\wb{\vv R}'_m\}_{m\in\N}$ satisfying a similar condition, and hence the sum of the two processes is also tight).
We have thus proved \eqref{eq:FCLT wt S}.
\end{proof}

To prove the desired convergence \eqref{eq:FCLT wb S}, we notice
\[
\nn{\wb{\vv S}_n-\wt{\vv S}_n}_\infty :=\max_{j=0,\dots,n}\abs{\wb S_{n,j}-\wt S_{n,j}} \le \frac 1{\sqrt n}|T_n-T_n'|\frac{a-1}{a+1},
\]
and therefore that
\equh\label{eq:tightness0}
\proba\pp{\nn{\wb{\vv S}_n-\wt{\vv S}_n}_\infty>\eta\mmid \calG} \le \proba\pp{\frac1{\sqrt n}|T_n-T_n'|\frac{a-1}{a+1}>\eta\mmid\calG}\to 0
\eque
as $n\to\infty$ for all $\eta>0$. Note that in the step \eqref{eq:tightness0} above we used the fact that $\{T_n-T_n'\}_{n\in\N}$ is conditionally tight given $\calG$, to be established in Lemma \ref{lemma:tightness} below.

\begin{remark}\label{rem:asw version}
In summary, we actually have proved a stronger version than the desired \eqref{eq:FCLT wb S}: thanks to \eqref{eq:tightness0} the almost surely weak convergence of \eqref{eq:FCLT wb S} (i.e.~\eqref{eq:FCLT wt S} with $\wt S_{n,{\floor {nt}}}'$ replaced by $\wb S'_{n,\floor{nt}}$) follows.
\end{remark}
\begin{lemma}\label{lemma:tightness}
The sequence of the conditional laws of $\{T_n'-T_n\}_{n\in\N}$ with respect to $\calG$ is tight. That is,
\[
\lim_{K\to\infty}\limsupn\proba(|T_n'-T_n|>K\mid\calG) = 0, \mbox{ almost surely.}
\]
\end{lemma}
\begin{proof}
We prove
\[
\lim_{K\to\infty}\limsupn\proba(T_n'-T_n>K\mid\calG) = 0, \mbox{ almost surely.}
\]
The other part (with $T_n'-T_n>K$ replaced by $T_n-T_n'>K$) follows essentially by the same proof by symmetry between the left and right processes in  Denisov's representation.

Recall that the centered random walk $(\what S_{n,m}')_{m=1,\dots,n}$ appears in the definition of $T_n'$ in \eqref{eq:T_n'}. Let
\[
\wt T_n':=\min \ccbb{j\in\{T_n,\dots n\}:\what S'_{n,j} = \min_{k=T_n,\dots,n}\what S'_{n,k}}
\]
denote the first location of the minimum of the transformed right process $\what {\vv S}'_n$ over the interval $\{T_n,\dots,n\}$. In particular, if $T_n'>T_n$, then $T_n' = \wt T_n'$. It then follows that
\equh\label{eq:T_n'-T_n}
\proba(T_n'-T_n>K\mid \calG)  \le \proba\pp{\wt T_n'-T_n>K\mmid\calG}.
\eque
Now, the event on the right-hand side concerns only the transformed right process
\equh\label{eq:transformed}
\pp{\what S_{n,T_n+j}'-\what S'_{n,T_n}}_{j=0,\dots,n-T_n}
\eque
 in Denisov's representation. To examine this process further, we recall the random variables $\{X_n\}_{n\in\N}, \{X_n'\}_{n\in\N}$ and the corresponding random walks $\{W_n\}_{n\in\N}$ and $\{W_n'\}_{n\in\N}$ from Section \ref{sec:aux}.  Recall also $\tau_{-1}:=\min\{j\in\N:W_j = -1\}$.
 Set also
\[
\what W_j' :=W_j' - \frac j2, j\in\N_0.
\]
The key observation is that the conditional law of \eqref{eq:transformed} given the event $T_n = m$ is the same as the conditional law of $(\what W_0',\dots,\what W_{n-m}')$ given $\tau_{-1}>n-m$.
Set also
\begin{align*}
T_m'' & :=\min\ccbb{j\in\{0,\dots,m\}:\what W_j' = \min_{k\in\{0,\dots,m\}}\what W_k'}, m\in\N,\\
T'' &:=\min\ccbb{j\in\N_0:\what W_j' = \min_{k\in\N_0}\what W_k'}.
\end{align*}
We have, for $m\in\N$ such that $m+K<n$ (otherwise the probability below is zero),
\begin{align*}
\proba\pp{\wt T_n' - T_n>K\mmid T_n = m} & = \proba\pp{T''_{n-m}>K\mmid \tau_{-1}>n-m}\\
& \le \proba\pp{T''>K\mmid \tau_{-1}>n-m} \le \frac1{\proba(\tau_{-1}>n-m)}\proba\pp{T''>K}\\
& \le \frac1{\proba(\tau_{-1} = \infty)}\proba\pp{T''>K}.
\end{align*}
We used the key observation in the first equality, the formula of conditional probability in the second inequality, and the fact that $\proba(\tau_{-1}>n-m)\le \proba(\tau_{-1} = \infty)$ in the last inequality. Plugging the above into \eqref{eq:T_n'-T_n}, we have
\[
\proba(T_n'-T_n>K\mid \calG)  \le\frac1{\proba(\tau_{-1} = \infty)}\proba(T''>K).
\]
Clearly, $\proba(\tau_{-1} = \infty)\in(0,1)$ and $T''$ is an almost surely finite random variable (since $a>1$ and hence the random walk $\what W_n'$ drifts to infinity).
The desired result now follows.
\end{proof}

\subsection*{Acknowledgements}
The authors thank two anonymous referees for helpful and constructive comments. The authors
 thank Ivan Corwin for encouraging and stimulating discussions.

\subsection*{Funding} W.B.~was partially supported by Simons Foundation~(703475). Y.W.~was partially supported by Simons Foundation (MP-TSM-00002359).
\subsection*{Data availability}
There is no data associated with this article.

\subsection*{Declarations}

\subsection*{Conflict of interest}
The authors have no relevant financial or non-financial interests to disclose.

\bibliographystyle{apalike}
\bibliography{references,references18}
\end{document}